\documentclass[8pt]{article}
\usepackage[a4paper, total={5.5in, 8in}]{geometry}
\usepackage[numbers]{natbib}
\usepackage{amsthm}
\usepackage{mathrsfs}
\usepackage[breaklinks,bookmarks=false]{hyperref}
\hypersetup{colorlinks, linkcolor=blue, citecolor=blue,
urlcolor=blue, plainpages=false, pdfwindowui=false,
pdfstartview={FitH}}

\usepackage{amsmath,amsfonts,amsthm,amssymb,mathtools,subcaption,enumitem,comment}
\usepackage[only,llbracket,rrbracket,llparenthesis,rrparenthesis]{stmaryrd} 
\usepackage[textsize=tiny,color=blue!10!white]{todonotes}
\usepackage{bm}
\usepackage{adjustbox}

\usepackage[medium,compact]{titlesec}
\titleformat*{\section}{\large\bfseries}
\titleformat*{\subsection}{\bfseries}

\newtheorem{theorem}{Theorem}[section]{\bfseries}{\it}
{\bfseries}{\it}
\newtheorem{lemma}[theorem]{Lemma}{\bfseries}{\it}
{\bfseries}{\it}
{\bfseries}{\it}
{\bfseries}{\rmfamily}
{\bfseries}{\it}
\theoremstyle{definition}
\newtheorem{remark}{Remark}[section]{\bfseries}{\rmfamily}
\newcommand{\proofbox}{\qed}

\numberwithin{equation}{section}

\usepackage{amsmath,amsfonts,amssymb,amsthm}
\usepackage{graphicx,verbatim,mathtools,xcolor}
\usepackage[only,llparenthesis,rrparenthesis]{stmaryrd}
\usepackage{bm}

\newcommand{\calP}{\mathcal{P}}

\newcommand{\calT}{\mathcal{T}}

\newcommand{\ver}{{\bm{a}}}
\newcommand{\dd}{\mathrm{d}}
\newcommand{\sdim}{d}
\newcommand{\abs}[1]{\lvert#1\rvert}
\newcommand{\tends}{\rightarrow}
\newcommand{\norm}[1]{\lVert#1\rVert}
\newcommand{\p}{\partial}

\DeclareMathOperator{\Div}{div}

\DeclareMathOperator{\diam}{diam}
\DeclareMathOperator*{\argmin}{arg\,min}
\newcommand{\pair}[2]{\langle #1,#2 \rangle}
\newcommand{\R}{\mathbb{R}}

\newcommand{\RTNK}{\bm{RTN}_{\widetilde{p}}(K)}
\newcommand{\RTNa}{\bm{RTN}_{\widetilde{p}}(\Ta)}

\newcommand{\Ver}{\mathcal{V}}
\newcommand{\Ta}{\calT^{\ver}}
\newcommand{\oma}{{\omega_{\ver}}}
\newcommand{\psia}{\psi_{\ver}}
\newcommand{\Qa}{Q_h^{\ver}}
\newcommand{\Va}{\bm{V}_h^\ver}
\newcommand{\etaOscEth}{\eta_{\mathrm{osc},h,\tau,E}}
\newcommand{\wetaOscEth}{\widetilde{\eta}_{\mathrm{osc},h,\tau,E}}
\newcommand{\etaOscEthGlobal}{\eta_{\mathrm{osc},h,\tau,E,\star}}
\newcommand{\Bz}{B_{\bm{Z}}}
\newcommand{\normYs}[1]{\norm{#1}_{Y}}

\newcommand{\T}{\mathcal{T}}
\newcommand{\VTp}{V_{h}}
\newcommand{\uht}{u_{h,\tau}}
\newcommand{\Uht}{U_{h,\tau}}
\newcommand{\ouht}{\overline{u}_{h,\tau}}
\newcommand{\Vhtp}{\mathbb{V}_{h,\tau}^+}
\newcommand{\sht}{\bm{\sigma}_{h,\tau}}

\title{On the efficiency of a posteriori error estimators for  parabolic partial differential equations in the energy norm}
\author{Iain Smears\footnotemark[2]}

\begin{document}

\maketitle

\begin{abstract}
For the model problem of the heat equation discretized by an implicit Euler method in time and a conforming finite element method in space, we prove the efficiency of \emph{a posteriori} error estimators with respect to the energy norm of the error, when considering the numerical solution as the average between the usual continuous piecewise affine-in-time and piecewise constant-in-time reconstructions.
This illustrates how the efficiency of the estimators is not only possibly dependent on the choice of norm, but also on the choice of notion of numerical solution.
\end{abstract}

\renewcommand{\thefootnote}{\fnsymbol{footnote}}

\footnotetext[2]{Department of Mathematics, University College London, Gower
Street, WC1E 6BT London, United Kingdom (\texttt{i.smears@ucl.ac.uk}).}

\section{Introduction}

The numerical approximation of partial differential equations (PDE) is crucial in many scientific and engineering fields. 
Ensuring the accuracy of the approximate solutions requires robust error control by \emph{a posteriori} error estimators, which provide computable bounds on the discretization error without requiring knowledge of the exact solution.
For time-dependent problems, such as parabolic PDE, there is a wide range of challenges that are not typically encountered for their steady-state counterparts.
As a model problem, we consider the heat equation 
\begin{equation}\label{eq:parabolic}
\begin{aligned}
\p_t u - \Delta u = f & & & \text{in }\Omega\times(0,T),\\
 u = 0 & & &\text{on }\partial\Omega\times (0,T),\\
 u(0) = u_0 & & &\text{in }\Omega,
\end{aligned}
\end{equation}
where $\Omega$ is a bounded, open polytopal domain in $\R^\sdim$, $\sdim \geq 1$, with Lipschitz boundary $\partial\Omega$, and where $T>0$ is the final time. We assume in the following that $f\in L^2(0,T;H^{-1}(\Omega))$ and that $u_0\in L^2(\Omega)$.
We will consider below, for simplicity, an implicit Euler discretization in time with a conforming finite element approximation in space.
One challenge that appears in the context of time dependent problems is that there is a plethora of norms that can be considered to measure the error, and the analytical properties of error estimators may vary depending on the choice of norm.
Early works~\cite{ErikssonJohnson1987a,ErikssonJohnson1991,ErikssonJohnson1995,JohnsonNieThomee1990}, as well as some more recent ones~\cite{DemlowLakkisMak2009,LakkisMakPryer2015,Sutton2020}, treat various norms such as $L^\infty(L^2)$ and $L^\infty(L^\infty)$, see also~\cite{Verfurth1998b} regarding the $L^2(L^2)$ norm.
Verf\"urth~\cite{Verfurth1998a} and Picasso~\cite{Picasso1998} obtained some of the first results on the efficiency of the estimators in the context of the $L^2(H^1)$ norm, under various restrictions on the relationship between the relative sizes of elements of the spatial mesh and the time-steps.
In~\cite{Verfurth1998a}, the efficiency of the estimators was shown under a condition of the form $\tau\simeq h^2$, where $h$ is the spatial mesh size and $\tau$ is the time-step size.
In~\cite{Picasso1998}, it was shown that the spatial estimator can be bounded by the $L^2(H^1)$-norm of the error \emph{plus the temporal jump estimator}, under the condition $\tau\simeq h$.
However, it is not yet known if the additional term of the temporal jump estimator can be removed from the right-hand side of these bounds.
Later in~\cite{Verfurth2003}, it was shown that the same estimators of~\cite{Verfurth1998a} also provide, up to data oscillation terms, upper bounds on the $L^2(H^1)\cap H^1(H^{-1})$ norm of the error, and moreover that the estimators are then efficient, locally-in-time yet globally-in-space, without any restrictions between mesh and time-step sizes. 
Equilibrated flux estimators were also analysed for this functional setting in~\cite{ErnVohralik2010}, covering a wide range of spatial discretization schemes.
We also note that various alternative error quantifiers (also called error measures in the literature) have been considered, e.g.\ of energy norm type~\cite{AkrivisMakridakisNochetto2009,MakridakisNochetto2006,Schotzau2010}, for semi-discrete approximations; these quantifiers typically combine simultaneously several reconstructions of the numerical solution, which we shall discuss further below.
More recently, a different functional setting was developed in~\cite{ESV2017}, based on the extension of the $L^2(H^1)\cap H^1(H^{-1})$ norm to the nonconforming-in-time approximation space, and equilibrated flux estimators were developed that give guaranteed upper bounds on the error, as well as showing efficiency bounds that are local in both space and time.
Then, in~\cite{ESV2019}, it was shown that the equilibrated flux estimators are also efficient in the $L^2(H^1)$-norm plus the temporal jump estimator under the improved one-sided condition $h^2\lesssim \tau$, which is the practically relevant case for computations.

Overall, it is clear from these references that the choice of norm bears a significant influence on the question of efficiency of the estimators.
Yet, as was observed in~\cite{ESV2017}, the analysis of efficiency of the estimators also is influenced by another issue, which is rather more subtle and easily-overlooked, namely the matter of choice in the \emph{notion of numerical solution}.
Indeed, for many time-stepping methods such as the implicit Euler method, there is an element of choice in how to reconstruct a discrete function that extends the computed values at the time-step points to the whole space-time domain.
For instance, \cite{Picasso1998,Verfurth2003} consider the numerical solution as a continuous piecewise affine-in-time function that interpolates the values at the time-step points; see also generalizations to higher-order temporal discretizations in~\cite{MakridakisNochetto2006}, whereas \cite{ESV2017} considered the solution as generally discontinuous-in-time function, in particular a piecewise constant-in-time approximation in the lowest-order case of an implicit Euler discretization.
In this work, we aim to investigate further the intricate relationship between the estimators, norms and also possible reconstructions.
We shall consider in particular the energy norm of the error
\begin{equation}\label{eq:energy_identity}
\norm{u- \widetilde{u}_{h,\tau} }_E^2\coloneqq \frac{1}{2}\norm{u(T) - \widetilde{u}_{h,\tau}(T)}_\Omega^2 + \int_0^T \norm{\nabla (u-\widetilde{u}_{h,\tau})}_\Omega^2\dd t,
\end{equation}
where $\norm{\cdot}_\Omega$ denotes the usual norm on $L^2(\Omega)$ or $L^2(\Omega;\R^\sdim)$, depending on the argument, and where the numerical approximation~$\widetilde{u}_{h,\tau}$ is some chosen reconstruction with regards to the temporal approximation.
To motivate the ideas in this work, let us examine more closely the role of the choice of reconstruction by considering momentarily the semi-discrete setting.

\paragraph{Role of the reconstruction in the semi-discrete setting.}
Let us consider, in a first instance, the semi-discrete approximation of~\eqref{eq:parabolic}, using an implicit Euler discretization in time, but without spatial discretization.
Let $\{t_n\}_{n=0}^N$, for some integer $N\geq 1$, denote a strictly increasing sequence of time-step points, with $t_0 = 0$ and $t_N = T$.
For each $n\in \{1,\dots,N\}$, let $I_n \coloneqq (t_{n-1},t_n)$ denote the corresponding time-step interval, and let $\tau_n = t_n - t_{n-1} >0$ denote the corresponding step length.
For the sake of simplicity in this introduction, let us momentarily assume that the source term $f$ is piecewise constant-in-time with respect to the time-step partition and that the initial datum $u_0\in H^1_0(\Omega)$.
For each $n\in\{1,\dots,N\}$, let $u_n \in H^1_0(\Omega)$ be defined inductively by 
\begin{equation}\label{eq:IE_as_FD}
\left( \frac{u_{n}-u_{n-1}}{\tau_n} , v \right)_{\Omega} + (\nabla u_n, \nabla v )_\Omega = \pair{f_n}{v} \quad \forall v \in H^1_0(\Omega),
\end{equation}
where $u_0$ is the initial datum from~\eqref{eq:parabolic}, where $f_n\coloneqq f|_{I_n}$, and where $(\cdot,\cdot)_\Omega$ denotes the usual $L^2$ inner-product for either scalar or vector fields, depending on the arguments.
Let $u_\tau\colon [0,T]\rightarrow H^1_0(\Omega)$ denote the unique left-continuous piecewise constant-in-time function given by $u_\tau|_{I_n}=u_n$ for all $n\in\{1,\dots,N\}$ as well as $u_\tau(0)=u_0$. 
Let $U_\tau\colon [0,T]\rightarrow H^1_0(\Omega)$ denote the continuous piecewise affine-in-time function that interpolates the $\{u_n\}_{n=0}^N$ at the time-step nodes. 
It is straightforward to check that $u_{\tau}(T)=U_{\tau}(T)=u_N$, and that $(\p_t U_\tau(t),v)_\Omega + (\nabla u_\tau(t),\nabla v)_\Omega = \pair{f(t)}{v}$ for all $v\in H^1_0(\Omega)$, where $\pair{\cdot}{\cdot}$ denotes the duality pairing between $H^1_0(\Omega)$ and its dual $H^{-1}(\Omega)$. 
This identity implies that
\begin{equation}\label{eq:semidisc_1}
  \pair{\p_t (u-U_\tau)(t)}{v}+(\nabla(u-u_\tau)(t),\nabla v)=0 \quad \forall v\in H^1_0(\Omega),
  \end{equation}
for a.e.\ $t\in (0,T)$. 
It is then found that the errors $u-u_\tau$ and $u-U_\tau$ are orthogonal in the inner-product~$(\cdot,\cdot)_E$ that relates to the energy norm, since
\begin{equation}\label{eq:orthogonality}
\begin{split}
(u-u_\tau,u-U_\tau)_E &\coloneqq
  \frac{1}{2}\left(u-u_{\tau}(T),u-U_{\tau}(T)\right)_\Omega+\int_0^T \left(\nabla(u-u_\tau),\nabla(u-U_\tau)\right)_{\Omega} \mathrm{d}t
  \\ &= \frac{1}{2}\norm{(u-U_\tau)(T)}_\Omega^2 - \int_0^T \pair{\p_t (u-U_\tau)}{u-U_\tau}\mathrm{d}t
  \\ & =0,
\end{split}
  \end{equation}
where in passing to the second line above we used $u_\tau(T)=U_\tau(T)$ and \eqref{eq:semidisc_1}, and in passing to the last line we used $u(0)=U_\tau(0)=u_0$. 
The orthogonality of the errors $u-u_\tau$ and $u-U_\tau$ implies the Pythagoras identity
\begin{equation}\label{eq:hypercircle}
\norm{u-u_\tau}_E^2 + \norm{u-U_\tau}_E^2 = \norm{u_\tau-U_\tau}_E^2,
\end{equation}
where, again using the fact that $u_\tau(T)=U_\tau(T)$, it is found that
\begin{equation}
\norm{u_\tau-U_\tau}_E^2= \int_0^T \norm{\nabla(u_\tau-U_\tau)}_\Omega^2\mathrm{d}t = \frac{1}{3}\sum_{n=1}^N \tau_n\norm{\nabla(u_n-u_{n-1})}_\Omega^2.
\end{equation}
Thus the estimator $\norm{u_\tau-U_\tau}_E$ is the well-known temporal jump estimator that appears in numerous analyses of \emph{a posteriori} analysis of parabolic problems.
The identity~\eqref{eq:hypercircle} is a parabolic analogue of the celebrated Prager--Synge identity~\cite{PragerSynge1947}, which plays a central role in the analysis of elliptic problems. 
In particular, it implies immediately the bounds $\norm{u-u_\tau}_E \leq \norm{u_\tau-U_\tau}_E$ and $\norm{u-U_\tau}_E\leq \norm{u_\tau-U_\tau}_E$.
However, the problem is that, in general, the temporal jump estimator $\norm{u_\tau-U_\tau}_E$ is not efficient with respect to either $\norm{u-u_\tau}_E$ or $\norm{u-U_\tau}_E$.
For the sake of completeness, we include a short example on this issue in Appendix~\ref{sec:counterexamples} below.
Yet, following Prager and Synge~\cite{PragerSynge1947}, it interesting to observe that~\eqref{eq:hypercircle} is equivalent to
\begin{equation}\label{eq:hypercircle_2}
  \norm{u-\overline{u}_\tau}_E = \frac{1}{2}\norm{u_\tau-U_\tau}_E, \qquad \overline{u}_\tau \coloneqq \frac{1}{2}(u_\tau+U_\tau),
\end{equation}
 which expresses the fact that, in the geometry determined by the energy norm, the three functions $u$, $u_\tau$, $U_\tau$ lie on a hypercircle of radius $\frac{1}{2}\norm{u_\tau-U_\tau}_E$ that is centred on $\overline{u}_\tau$, the midpoint between $u_\tau$ and $U_\tau$.
Note that considering the centre of the hypercircle as the appropriate notion of numerical solution rejoins an idea that was already proposed by Prager and Synge~\cite[p.~248]{PragerSynge1947}. 
As we show below, it turns out that this third option offers an effective generalization to the fully discrete setting.

\paragraph{Contributions.} 

We show that~\eqref{eq:hypercircle_2} can be generalized appropriately to the fully discrete setting, leading to upper and lower bounds on the error when considering the numerical solution as the midpoint between the two standard reconstructions. To be clear, we are not suggesting that this notion of numerical solution needs to be used in practical computations over the alternatives, but rather that it helps to better understand the behaviour of the estimators.

Our main results are as follows. 
For fully discrete approximation based on an implicit Euler method in time and conforming finite element method in space, resulting in a piecewise constant-in-time approximation $\uht$ and continuous piecewise affine approximation $\Uht$, we consider error bounds for $u-\ouht$ where $\ouht = \frac{1}{2}(\uht+\Uht)$.
We show, in Theorem~\ref{thm:energy_main_bound} below, an upper bound of the form
\begin{equation}
\norm{u-\ouht}_E^2 \leq \int_0^T \left(\frac{1}{4}\norm{\nabla(\uht-\Uht)}_\Omega^2 + \norm{\sht+\nabla \ouht}_{\Omega}^2\right)\mathrm{d}t + \mathrm{oscillation},
\end{equation}
where $\sht\in L^2(0,T;H(\Div,\Omega))$ is a locally computable equilibrated flux that is constructed in a similar manner as in~\cite{ESV2017,ESV2019}.
Furthermore, in Theorem~\ref{thm:lower_bound} below, we obtain a global lower bound of the form
\begin{equation}
\int_0^T \left(\frac{1}{4}\norm{\nabla(\uht-\Uht)}_\Omega^2 + \norm{\sht+\nabla \ouht}_{\Omega}^2\right)\mathrm{d}t \lesssim \norm{u-\ouht}_E^2 + \mathrm{oscillation},
\end{equation}
where the hidden constant is independent of the discretization and problem parameters, under two additional hypotheses which we discuss more below. However, a crucial point that we wish to clarify immediately is that the efficiency of the estimators does not derive from the fact that one is considering the term $\sht+\nabla\ouht$ in the flux estimator instead of either $\sht+\nabla \uht$ or $\sht+\nabla\Uht$ since all of these would result in equivalent total estimators by virtue of the triangle inequality, so efficiency also holds for all of these possible choices.
The additional hypotheses used to show the efficiency bound include the assumption that the $L^2$-projection operator into the finite element space is $H^1$-stable, and also a one-sided condition that the local mesh-size $h^2\lesssim \tau$ the time-step size.
Note that this latter condition enables large time-steps and is relevant to practical computations. We remark also that the $H^1$-stability of the $L^2$-projection is known to play an important role in numerical methods for parabolic problems~\cite{TantardiniVeeser2016}, and also that this assumption is verified for a range of mesh refinements in practice~\cite{GaspozSiebert2016}.
As a significant original ingredient for the analysis,  we derive a framework that characterizes the energy norm in terms of an inf-sup identity for a suitable weak form of the heat equation. 
This allows us to overcome various difficulties that would otherwise appear when trying the more usual approach to energy norm bounds, which is based on testing the residual with the error.

For the sake of simplicity, we do not go into detail into possible extensions of these results to other settings. First, we focus here only on the case of equilibrated flux estimators, since we make use of results from~\cite{ESV2019}. 
However, other spatial error estimators could most likely also be considered.
One can also consider other time-stepping methods, such as higher-order approximations in time, following the ideas in~\cite{ESV2017,ESV2019}. Yet another further extension involves mesh-modification between time-steps, as analysed in the previous works~\cite{ESV2017,ESV2019}, since this plays an important role in adaptive methods~\cite{ChenFeng2004,Dupont1982,ErikssonJohnson1991,GaspozSiebertKreuzerZiegler2019,Kreuzer2012}.

\section{Notation and setting}

Let $H^1_0(\omega)$ denote the closure of $C^\infty_0(\omega)$ in the space $H^1(\omega)$, where $C^\infty_0(\omega)$ denotes the space of real-valued infinitely differentiable compactly supported functions on $\omega$.
Note that for $\omega$ bounded, the Poincar\'e inequality implies that the mapping $v\mapsto \norm{\nabla v}_{\omega}$ defines an equivalent norm on $H^1_0(\omega)$, see \cite[Corollary 6.31, p.~184]{AdamsFournier2003}. 
Let $H^{-1}(\omega)$ denote the dual space of $H^1_0(\omega)$, with norm
\begin{equation}\label{eq:dual_norm}
\norm{\Phi}_{H^{-1}(\omega)}\coloneqq \sup_{v\in H^1_0(\omega)\setminus\{0\}}\frac{\pair{\Phi}{v}_{H^{-1}(\omega)\times H^1_0(\omega)}}{\norm{\nabla v}_\omega} \quad \forall \Phi\in H^{-1}(\omega),
\end{equation}
where $\pair{\cdot}{\cdot}_{H^{-1}(\omega)\times H^1_0(\omega)}$ is the duality pairing between $H^{-1}(\omega)$ and $H^1_0(\omega)$.
As noted above, in the case $\omega=\Omega$, we drop the subscript and denote the duality pairing simply by $\pair{\cdot}{\cdot}$.
The space $L^2(\omega)$ can be canonically embedded into $H^{-1}(\omega)$ through $\pair{w}{ v}_{H^{-1}(\omega)\times H^1_0(\omega)} = (w,v)_{\omega}$ for all $v\in H^1_0(\omega)$. 
Thus the spaces $H^1_0(\omega)$, $L^2(\omega)$ and $H^{-1}(\omega)$ form a \emph{Gelfand triple}, with $H^1_0(\omega)\subset L^2(\omega)\subset H^{-1}(\omega)$, where each embedding is continuous, dense and injective.

The analysis will be formulated in terms of the following function spaces. First, let
\begin{equation}\label{eq:XY_spaces_def_1}
\begin{aligned}
  X & \coloneqq L^2(0,T;H^1_0(\Omega)).
\end{aligned}
\end{equation}
Since $\Omega$ is bounded, we shall equip $X$ with the norm $\norm{\cdot}_X$ defined by
\begin{equation}\label{eq:X_norm}
\norm{v}_X^2 \coloneqq \int_0^T \norm{\nabla v}_\Omega^2 \mathrm{d}t \quad\forall v \in X.
\end{equation}
Next, let 
\begin{equation}\label{eq:Y_def}
Y \coloneqq L^2(0,T;H^1_0(\Omega))\cap H^1(0,T;H^{-1}(\Omega)).
\end{equation}
It is known that $Y$ is continuously embedded in $C([0,T];L^2(\Omega))$, see~\cite[p.~287]{Evans1998}.
Thus, we may consider the norm $\normYs{\cdot}$ on the space $Y$ defined by
\begin{equation}\label{eq:Ys_def}
\begin{aligned}
\normYs{\varphi}  \coloneqq \int_0^T\left(
\norm{\partial_t \varphi}_{H^{-1}(\Omega)}^2+\norm{\nabla\varphi}_\Omega^2\right)\dd t + \norm{\varphi(0)}_\Omega^2+\norm{\varphi(T)}_\Omega^2  &&& \forall\varphi \in Y,
\end{aligned}
\end{equation}
We remark that $\normYs{\cdot}$ is then invariant with respect to reversal of the time variable, i.e.\ the map $Y\ni\varphi\mapsto \varphi(T-\cdot)$ is an isometry under the norm $\normYs{\cdot}$. 
We also let $\bm{Z}\coloneqq X\times L^2(\Omega) $ denote the product of space of $X$ with $L^2(\Omega)$.
A norm on the space $\bm{Z}$ is defined by
\begin{equation}\label{eq:Z_norm}
\begin{aligned}
\norm{\bm{v}}_{\bm{Z}}^2 &\coloneqq \norm{v}_X^2+\frac{1}{2}\norm{v_T}_\Omega^2  && \forall\bm{v}=(v,v_T)\in \bm{Z}.
\end{aligned}
\end{equation}
Note that the choice of a factor of $\frac{1}{2}$ before the term $\norm{v_T}_\Omega^2$ in~\eqref{eq:Z_norm} is motivated by the energy norm.

\paragraph{Inf-sup stability and connection to the energy norm.}

The following Lemma gives a useful alternative formula for the norm~$\normYs{\cdot}$ that was defined in~\eqref{eq:Ys_def} above. See also~\cite{UrbanPatera2012,ESV2017} for similar identities.
\begin{lemma}\label{lem:Ys_identity}
We have the identity
\begin{equation}\label{eq:Ys_norm_identity}
\normYs{\varphi}^2 = 2\norm{\varphi(T)}_\Omega^2 + \int_0^T\norm{(\p_t + \Delta) \varphi}_{H^{-1}(\Omega)}^2 \dd t \quad \forall \varphi \in Y.
\end{equation}
\end{lemma}
Note that in~\eqref{eq:Ys_norm_identity}, $\Delta \colon H^1_0(\Omega)\rightarrow H^{-1}(\Omega)$ denotes the Laplacian operator understood in a distributional sense.

\begin{proof}
Let $\varphi\in Y$ be arbitrary.
Let $z\in X$ be the unique function that satisfies $(\nabla z(t),\nabla v)_\Omega = \pair{\p_t \varphi(t)}{v}$ for all $v\in H^1_0(\Omega)$ and a.e.\ $t\in (0,T)$. 
It follows that $\norm{\nabla z}_\Omega^2 = \norm{\p_t \varphi}_{H^{-1}(\Omega)}^2$ and $\norm{\p_t\varphi + \Delta \varphi}_{H^{-1}(\Omega)}^2=\norm{\nabla(z-\varphi)}^2_\Omega$ a.e.\ in $(0,T)$.
So, by expanding the square, we get
\begin{equation}
\begin{split}
\int_0^T \left(\norm{\p_t\varphi + \Delta \varphi}_{H^{-1}(\Omega)}^2\right)\mathrm{d}t
 &= \int_0^T \left( \norm{\nabla z}_\Omega^2 - 2(\nabla z,\nabla \varphi)_\Omega + \norm{\nabla \varphi}_\Omega^2 \right)\mathrm{d}t 
\\ & = \int_0^T \left( \norm{\nabla z}_\Omega^2 - 2\pair{\p_t \varphi}{ \varphi} + \norm{\nabla \varphi}_\Omega^2\right) \mathrm{d}t 
\\ &= \int_0^T \left(\norm{\p_t \varphi }_{H^{-1}(\Omega)}^2 + \norm{\nabla \varphi}_\Omega^2\right) \mathrm{d}t + \norm{\varphi(0)}_\Omega^2 - \norm{\varphi(T)}_\Omega^2
\\ &= \normYs{\varphi}^2 - 2 \norm{\varphi(T)}_\Omega^2,
\end{split}
\end{equation}
where in the penultimate line we used the identity $\int_0^T 2\pair{\p_t \varphi}{\varphi}\mathrm{d}t = \norm{\varphi(T)}_\Omega^2-\norm{\varphi(0)}_\Omega^2 $.
\end{proof}

Define the bilinear form $\Bz:\bm{Z}\times Y\rightarrow \R$ by
\begin{equation}\label{eq:B_bilinear}
\begin{aligned}
\Bz(\bm{v},\varphi)\coloneqq (v_T,\varphi(T))_\Omega+ \int_0^T\left(-\pair{\partial_t \varphi}{v} + (\nabla \varphi, \nabla v )_\Omega \right) \mathrm{d}t  ,
\end{aligned}
\end{equation}
for all $\bm{v}=(v,v_T)\in \bm{Z}$ and all $\varphi\in Y$.

The following theorem states the inf-sup identities for the bilinear form $\Bz$.
\begin{theorem}[Inf-sup identities]\label{thm:Z_infsup}
We have the identities
\begin{subequations}\label{eq:infsup}
\begin{align}
\sup_{\varphi\in Y\setminus\{0\}}\frac{\Bz(\bm{v},\varphi)}{\normYs{\varphi}} &= \norm{\bm{v}}_{\bm{Z}} \quad\forall \bm{v}\in \bm{Z}, \label{eq:Z_infsup_1} \\
 \sup_{\bm{v}\in \bm{Z}\setminus\{0\}}\frac{\Bz(\bm{v},\varphi)}{\norm{\bm{v}}_{\bm{Z}}} &= \normYs{\varphi} \quad \forall \varphi\in Y. \label{eq:Z_infsup_2}
\end{align}
\end{subequations}
\end{theorem}
\begin{proof}
We start by showing that
\begin{equation}\label{eq:infsup_upper_bound}
\abs{\Bz(\bm{v},\varphi)}\leq \norm{\bm{v}}_{\bm{Z}}\normYs{\varphi}\quad \forall \bm{v}\in \bm{Z},\;\forall \varphi\in Y.
\end{equation}
It follows from the definition of the $H^{-1}(\Omega)$-norm in~\eqref{eq:dual_norm} that\begin{equation}
\begin{split}
\left\lvert\int_0^T \left(-\pair{\partial_t \varphi}{v} + (\nabla \varphi, \nabla v )_\Omega \right) \dd t\right\rvert
&=\left\lvert\int_0^T -\pair{\p_t \varphi + \Delta \varphi}{v} \dd t\right\rvert
\\ &\leq \left(\int_0^T \norm{\p_t \varphi + \Delta \varphi}_{H^{-1}(\Omega)}^2\dd t\right)^{\frac{1}{2}}\norm{v}_X,
\end{split}
\end{equation}
for any $\varphi\in Y$ and $v\in X$, where we recall that $\norm{v}_X=\left(\int_0^T\norm{\nabla v}_\Omega^2\dd t\right)^{\frac{1}{2}}$.
Hence, the Cauchy--Schwarz inequality applied to the terms in $\Bz(\bm{v},\varphi)$ shows that
\begin{equation}
\begin{split}
\abs{\Bz(\bm{v},\varphi)} &\leq \left(\frac{1}{2}\norm{v_T}_
\Omega^2 + \norm{v}_X^2\right)^{\frac{1}{2}}
\left(2\norm{\varphi(T)}_\Omega^2+\int_0^T\norm{ \partial_t \varphi + 
\Delta \varphi}_{H^{-1}(\Omega)}^2\mathrm{d}t\right)^{\frac{1}{2}}
\\& =\norm{\bm{v}}_{\bm{Z}}\normYs{\varphi},
\end{split}
\end{equation}
where we have used Lemma~\ref{lem:Ys_identity} in passing to the second line above. This yields~\eqref{eq:infsup_upper_bound}.
Next, let $\bm{v}=(v,v_T)\in \bm{Z}$ be arbitrary, and let $\varphi_*\in Y$ be the unique solution of the backward parabolic problem: find $\varphi_* \in Y$ such that $\varphi_*(T)=\frac{1}{2}v_T$ and $\partial_t \varphi_* + \Delta \varphi_* =  \Delta v$ in $(0,T)$. 
Then, it is clear that $\Bz(\bm{v},\varphi)= \frac{1}{2}\norm{v_T}_\Omega^2+\norm{v}_X^2 = \norm{\bm{v}}_{\bm{Z}}^2$. Furthermore, using Lemma~\ref{lem:Ys_identity}, we have
\begin{equation}
\normYs{\varphi_*}^2 = 2\norm{\varphi(T)}_\Omega^2 + \int_0^T \norm{\partial_t \varphi+\Delta \varphi}_{H^{-1}(\Omega)}^2\mathrm{d}t = \frac{1}{2}\norm{v_T}_
\Omega^2 + \norm{v}_X^2 = \norm{\bm{v}}_{\bm{Z}}^2.
\end{equation}
Thus we obtain \eqref{eq:Z_infsup_1} from the above identities and the upper bound~\eqref{eq:infsup_upper_bound}.
To show~\eqref{eq:Z_infsup_2}, for a given $\varphi\in Y$, we take $\bm{v}_*=(v,v_T)$ with $v=(-\Delta)^{-1}(-\partial_t \varphi) + \varphi \in X$, and $v_T = 2\varphi(T)\in L^2(\Omega)$, and we perform similar computations to find that $\norm{\bm{v}}_{\bm{Z}}=\normYs{\varphi}$ and $\Bz(\bm{v},\varphi)=\normYs{\varphi}^2$. This shows~\eqref{eq:Z_infsup_2}.
\end{proof}

We remark that we use the terminology of inf-sup identity in reference to the inf-sup condition~\cite[Theorem~3.2]{Necas1962}, since Theorem~\ref{thm:Z_infsup} implies that
\begin{equation*}
\inf_{\bm{v}\in \bm{Z}\setminus\{0\}}\sup_{\varphi\in Y\setminus\{0\}}\frac{\Bz(\bm{v},\varphi)}{\norm{\bm{v}}_{\bm{Z}}\normYs{\varphi}} = \sup_{\bm{v}\in \bm{Z}\setminus\{0\}}\sup_{\varphi\in Y\setminus\{0\}}\frac{\Bz(\bm{v},\varphi)}{\norm{\bm{v}}_{\bm{Z}}\normYs{\varphi}} =1,
\end{equation*}
i.e.\ the upper and lower constants in the usual inf-sup condition both coincide and equal one.

\paragraph{Heat equation.}

Under the hypotheses that $f\in L^2(0,T;H^{-1}(\Omega))$ and $u_0\in L^2(\Omega)$, it is well-known that \eqref{eq:parabolic} admits a unique solution $u\in Y$, see~\cite{LionsMagenes1972}.
We now show how to apply Theorem~\ref{thm:Z_infsup} to the heat equation.
After testing the heat equation with a test function $\varphi\in Y$ and integrating-by-parts, we see that the solution $u$ solves
\begin{equation}\label{eq:Z_weakform}
(u(T),\varphi(T))_{\Omega} + \int_0^T \left(-\pair{\p_t \varphi}{u} + (\nabla u,\nabla \varphi)_\Omega \right)\dd t = \int_0^T \pair{f}{\varphi}\dd t + (u_0,\varphi(0))_\Omega \quad \forall \varphi\in Y.
\end{equation}
Upon defining $\bm{u}\coloneqq(u,u(T))\in \bm{Z}$, it follows that~\eqref{eq:Z_weakform} can be written equivalently as 
\begin{equation}\label{eq:Z_weakform_2}
\Bz(\bm{u},\varphi)=\int_0^T \pair{f}{\varphi}\dd t + (u_0,\varphi(0))_\Omega \quad \forall \varphi\in Y.
\end{equation}
Hence~\eqref{eq:Z_infsup_1} implies that the energy norm $\norm{u}_E$ of the solution satisfies
\begin{equation}\label{eq:Z_norm_solution}
\norm{u}_{E} 
 = \norm{\bm{u}}_{\bm{Z}}  =  \sup_{\varphi\in Y\setminus\{0\}}\frac{\int_0^T\pair{f}{\varphi}\dd t + (u_0,\varphi(0))_\Omega}{\normYs{\varphi}} .
\end{equation}
Observe that this characterization provides an alternative to the usual energy identity based on testing \eqref{eq:parabolic} with the solution $u$, with the advantage of characterizing the energy norms in terms of a dual norm of the data.
Thus, when it comes to the \emph{a posteriori} error analysis below, we shall obtain an equivalence between the energy norm and a (generally noncomputable) dual norm of the residual, which can be bounded by the computable \emph{a posteriori} error estimators.

\subsection{Numerical Scheme}
We now consider the numerical approximation of the solution of \eqref{eq:parabolic} by a conforming finite element method on a fixed spatial mesh coupled with the implicit Euler discretization in time.
Let $\{t_n\}_{n=0}^N$, for some integer $N\geq 1$, denote a strictly increasing sequence of time-step points, with $t_0 = 0$ and $t_N = T$.
For each $n\in \{1,\dots,N\}$, let $I_n \coloneqq (t_{n-1},t_n)$ denote the corresponding time-step interval, and let $\tau_n = t_n - t_{n-1} >0$ denote the corresponding step length.
Let $\T$ be a conforming simplicial mesh on $\Omega$.
For each element $K\in\T$, we let $h_K$ denote the diameter of $K$. 
Let $p\geq 1$ denote a fixed integer which will correspond to the polynomial degree of the finite element space to be defined below.
The $H^1_0(\Omega)$-conforming finite element space of order $p$ is denoted by $\VTp$, and is defined by
\begin{equation}\label{eq:VTp_def}
\VTp \coloneqq \{ v \in H^1_0(\Omega), \; v|_K \in \calP_{p}(K) \quad\forall K\in\T \},
\end{equation}
where $\calP_p(K)$ denotes the space of real-valued polynomials of total degree at most $p$ on $K$.

The analysis below will make use of the constant determining the $H^1$-stability of the $L^2$-orthogonal projection operator $\Pi_h\colon H^1_0(\Omega)\rightarrow \VTp$.
Let $C_{\Pi}$ denote the optimal constant such that
\begin{equation}\label{eq:L2_stab_ortho}
\norm{\nabla \Pi_h v}_{\Omega} \leq C_{\Pi}\norm{\nabla v}_\Omega \qquad \forall v\in H^1_0(\Omega).
\end{equation}
The fact that a constant $C_\Pi$ exists is simply a consequence of the finite dimensionality of $\VTp$ and the Poincar\'e inequality for functions in $H^1_0(\Omega)$.
It is known that the constant $C_\Pi$ is independent of the mesh-size for a wide range of meshes used in practice, see~\cite{GaspozSiebert2016} and the references therein.
Note that \eqref{eq:L2_stab_ortho} implies the following bound: 
\begin{equation}\label{eq:DiscreteNegnorm_stab}
\norm{w_h}_{H^{-1}(\Omega)}\leq C_{\Pi} \norm{w_h}_{\VTp^*} \quad \forall w_h\in \VTp,
\end{equation}
where $\norm{w_h}_{\VTp^*}\coloneqq \sup_{v_h\in\VTp\setminus\{0\}}\frac{(w_h,v_h)_\Omega}{\norm{\nabla v_h}_\Omega}$ is the discrete dual norm.
Conversely, we have $\norm{w_h}_{\VTp^*}\leq \norm{w_h}_{H^{-1}(\Omega)}$ trivially, so necessarily $C_{\Pi}\geq 1$.

In the following, the notation for inequalities $a \lesssim b$ means $a\leq C b$ for some constant $C$ that is independent of the mesh and time-step sizes, but may depend on $\Omega$, $\sdim$, the shape-regularity parameter of $\T$, defined as $\theta_\T\coloneqq \max_{K\in\T}\frac{h_K}{\rho_K}$, where $\rho_K$ denotes the diameter of the largest inscribed balls in $K$, and also on the constant $C_\Pi$ of~\eqref{eq:L2_stab_ortho}. Thus, we are essentially supposing that $C_\Pi$ is independent of the mesh size.

\paragraph{Numerical method.}
The numerical method consists of solving the following sequence of discrete problems: for each $n\in \{1,\dots, N\}$, find $u_{h,\tau,n}\in \VTp$ such that
\begin{equation}\label{eq:IE_FEM}
\begin{aligned}
\left(\frac{u_{h,\tau,n}-u_{h,\tau,n-1}}{\tau_n}, v_{h}\right)_\Omega + (\nabla u_{h,\tau,n},\nabla v_h)_\Omega = (f_{h,\tau,n},v_{h})_\Omega  &&& \forall v_h \in \VTp,
\end{aligned}
\end{equation}
where $u_{h,\tau,0} \in \VTp$ is some discrete approximation of the initial datum $u_0$, and where $f_{h,\tau,n}\in L^2(\Omega)$ is some discrete approximation of $f(t_n)$.
We do not prescribe here a specific choice of the data approximations $u_{h,\tau,0}$ and $f_{h,\tau,n}$.
Instead, we allow for rather general approximations, with the only assumptions being that $u_{h,\tau,0}\in \VTp$ as stated above, and also that $f_{h,\tau,n}|_K \in \mathcal{P}_{p}(K)$ for each $K\in \T$, for each $n\in\{1,\dots,N\}$.
We remark here that for many standard choices, the various data oscillation terms that appear further below can be bounded via \emph{a priori} estimates, see for instance~\cite[Lemma~6.2]{ESV2019}.

\paragraph{Piecewise-constant-in-time reconstruction.}

Let $\Vhtp$ denote the space of all left-continuous functions $v\colon [0,T]\rightarrow \VTp$ that are piecewise constant with respect to the time partition $\{I_n\}_{n=1}^N$, i.e.\ $v|_{I_n}\in \mathcal{P}_0(I_n;\VTp) $ for each $n\in \{1,\dots,N\}$, where  we recall that $I_n=(t_{n-1},t_n)$ is the $n$-th time-interval.
Note that a function  $v\in \Vhtp$ thus has a well-defined value $v(t)\in \VTp$ for each time $t\in[0,T]$. Note that from this point of view, functions in $\Vhtp$ that agree for almost all $t\in [0,T]$ are not identified.
Let $\uht\in \Vhtp$ be the unique function that satisfies
\begin{equation}
\begin{aligned}
\uht(0)= u_{h,\tau,0}, &&& \uht|_{I_n} =u_{h,\tau,n} \quad \forall n\in\{1,\dots,N\}.
\end{aligned}
\end{equation}
Observe also that left-continuity of $\uht$ ensures that $\uht(t_n) = u_{h,\tau,n}$ for each $n\in \{0,\dots, N\}$.

\paragraph{Continuous piecewise-affine-in-time reconstruction.}
Let $\Uht \in C([0,T];\VTp)$ as the unique continuous piecewise-affine function that satisfies
\begin{equation}\label{eq:fd_affine_recons}
\begin{aligned}
\Uht(t)= \frac{t-t_{n-1}}{\tau_n} u_{h,\tau,n} + \frac{t_n-t}{\tau_n} u_{h,\tau,n-1} &&& \forall t \in [t_{n-1},t_n].
\end{aligned}
\end{equation}
Note that, by definition, $\Uht(t_n)=u_{h,\tau,n}$ for each $n\in \{0,\dots,N\}$.

\paragraph{Reformulation of the fully discrete scheme.}
Let the piecewise constant-in-time function $f_{h,\tau} \in L^2(0,T;L^2(\Omega))$ be defined by $f_{h,\tau}|_{I_n} \coloneqq f_{h,\tau,n}$ for $n\in \{1,\dots,N\}$, where $f_{h,\tau,n}$ is the chosen approximation of $f(t_n)$ appearing in \eqref{eq:IE_FEM} above.
Under the assumptions above on the $f_{h,\tau,n}$, it follows that $f_{h,\tau}$ is not only piecewise constant in time with respect to the time-intervals $\{(t_{n-1},t_n)\}_{n=1}^N$, but also a piecewise polynomial of degree at most $p$ in space with respect to the mesh $\T$.
It is then clear that \eqref{eq:IE_FEM} implies that 
\begin{equation}\label{eq:IE_FEM_2}
  ( \p_t \Uht (t) , v_h )_\Omega +  (\nabla \uht(t) , \nabla v_h)_\Omega = (f_{h,\tau}(t),v_h)_\Omega  \quad \forall v_h\in \VTp, \quad \text{a.e.\ } t\in (0,T).
\end{equation}

\subsection{Construction of the equilibrated flux}

We consider now \emph{a posteriori} error bounds for the error based on the construction of an equilibrated flux.
The construction follows the one in~\cite{ESV2017,ESV2019}, which handles variable polynomial degrees in $hp$-FEM and also mesh modification between the time-steps. 
Since we are restricting ourselves here to the case of lowest-order approximations in time, with a single fixed polynomial degree across the spatial mesh elements, and also a fixed mesh for all time-steps, we shall give here a slightly simplified construction relative to the one in~\cite{ESV2017,ESV2019}.

Let $\Ver$ denote the set of vertices of the mesh $\T$.
Let $\Ver_{\Omega}=\Ver\cap \Omega$ denote the set of interior vertices and let $\Ver_{\partial\Omega}=\Ver\cap \partial\Omega$ denote the boundary vertices.
For each vertex $\ver\in\Ver $, let $\psia$ denote the corresponding hat function.
Note that $\{\psia\}_{\ver\in\Ver}$ forms a partition of unity of $\overline{\Omega}$, i.e. $\sum_{\ver\in\Ver}\psia(x)=1$ for all $x\in \overline{\Omega}$. 
Let $\Ta$ denote the set of all elements of $\T$ that contain $\ver$, and let $\oma$ denote the vertex patch around $\ver$, i.e.\ the union of all elements of $\Ta$.
With the convention that elements of $\T$ are considered to be closed sets, it follows that $\oma$ is the support of $\psia$.
Let $\widetilde{p}\geq p+1$ denote a fixed choice of polynomial degree that will be used for the flux reconstruction, where it is recalled that $p$ is the polynomial degree used for $\VTp$ above. 
For each $\ver\in \Ver$, let the spaces $\calP_{\widetilde{p}}(\Ta)$ and $\RTNa$ be defined by
\begin{subequations} 
\begin{align}
\calP_{\widetilde{p}}(\Ta) &\coloneqq \{ q_h \in L^2(\oma)\colon q_h|_K \in \calP_{\widetilde{p}}(K)\quad\forall K\in\Ta\},\\
\RTNa &\coloneqq \{ \bm{v}_h \in L^2(\oma;\R^\sdim) \colon \bm{v}_h|_K \in \RTNK \quad\forall K\in \Ta \},
\end{align}
\end{subequations}
where $\RTNK\coloneqq \calP_{\widetilde{p}}(K;\R^\sdim)+\bm{x}\calP_{\widetilde{p}}(K)$ denotes the RTN space of order $\widetilde{p}$ on $K$. In other words, $\calP_{\widetilde{p}}(\Ta)$ denotes the space of scalar functions that are piecewise polynomials of degree at most~$\widetilde{p}$ with respect to~$\Ta$, and~$\RTNa$ denotes the space of vector fields that are piecewise RTN of order $\widetilde{p}$ with respect to~$\Ta$. 
Note that there are no continuity conditions across mesh elements in the definition of these spaces.
We now define the local mixed finite element spaces~$\Qa$ and~$\Va$ by
\begin{subequations}
\begin{align}
\Qa &\coloneqq \begin{cases}
\left\{q_h \in \calP_{\widetilde{p}}(\Ta), \quad \int_\oma q_h \mathrm{d}x = 0 \right\} &\text{if }\ver\in\Ver_{\Omega} ,
\\
\calP_{\widetilde{p}}(\Ta) &\text{if }\ver\in\Ver_{\partial\Omega},
\end{cases}\label{eq:Qa_def}
\\
\Va &\coloneqq 
\begin{cases}
\left\{\bm{v}_h \in H(\Div,\oma) \cap \RTNa\, \quad \bm{v}_h\cdot \bm{n} =0 \text{ on } \partial \oma \right\} &\text{if } \ver\in \Ver_{\Omega},
\\
\left\{\bm{v}_h \in H(\Div,\oma) \cap \RTNa\, \quad \bm{v}_h\cdot \bm{n} =0 \text{ on } \partial \oma \setminus \partial \Omega \right\} &\text{if } \ver\in \Ver_{\partial\Omega},
\end{cases}
\end{align}
\end{subequations}
where $\bm{n}$ denotes the unit outward normal on~$\partial \oma$.
For each $\ver\in \Ver$ and each $n\in {1,\dots,N}$, define the piecewise polynomial function $g_{h,\tau}^{\ver,n}\colon \oma \rightarrow \R$ by
\begin{equation}
g_{h,\tau}^{\ver,n} \coloneqq  \psia f_{h,\tau}|_{\oma\times I_n} - \psia \p_t \Uht|_{\oma\times I_n} - \nabla\psia{\cdot}\nabla \uht|_{\oma\times I_n}.
\end{equation}
Note in particular that the choice $\widetilde{p}\geq p+1$ is used above to ensure that the terms $\psia f_{h,\tau,n}$ and $\psia \p_t \Uht$ are piecewise polynomials of degree less than or equal to $\widetilde{p}$ with respect to the spatial mesh~$\T$. Therefore $g_{h,\tau}^{a,n} \in \calP_{\widetilde{p}}(\Ta)$ for each $\ver\in\Ver$.
Note also that in the case of an interior vertex $\ver \in \Ver_{\partial\Omega}$, choosing as a test function $v_h=\psia$ in~\eqref{eq:IE_FEM_2} above implies that
\begin{equation}\label{eq:source_compatibility_condition}
\int_{\oma} g_{h,\tau}^{\ver,n} \dd x = (f_{h,\tau,n},\psia)_\Omega - (\p_t\Uht,\psia)_\Omega - (\nabla \uht,\nabla \psia)_\Omega =0,
\end{equation}
where we recall that $f_{h,\tau}|_{I_n}=f_{h,\tau,n}$ by definition.
Therefore, $g_{h,\tau}^{\ver,n}$ has zero-mean value in the case of an interior vertex $\ver$, and thus we conclude that that $g_{h,\tau}^{\ver,n}\in \Qa$ for all vertices $\ver\in \Ver$, where it is recalled that $\Qa$ was defined in~\eqref{eq:Qa_def} above.
Next, for each time-step $n\in\{1,\dots,N\}$ and each vertex $\ver\in \Ver$, we define the local flux contribution $\sht^{\ver,n} \in \Va$ by 
 \begin{equation}
\sht^{\ver,n}\coloneqq \argmin_{\substack{\bm{v}_{h,\tau}\in \Va \\ \nabla\cdot\bm{v}_{h,\tau} = g_{h,\tau}^{\ver,n}}} \norm{\bm{v}_{h,\tau}+\psia\nabla u_{h,\tau}|_{I_n}}_{\oma}.
\end{equation}
Crucially, the fact that $\sht^{\ver,n}$ is well-defined in the case of an interior vertex is a consequence of~\eqref{eq:source_compatibility_condition}.
Next, we extend the individual local contributions by zero to the whole domain $\Omega$. Note that boundary conditions imposed on vector fields in $\Va$ imply that the extension by zero of $\sht^{\ver,n}$ to the rest of $\Omega$, with same notation, is in $H(\Div,\Omega)$.
Therefore, we may define the global flux $\sht\in L^2(0,T;H(\Div,\Omega))$ as the unique piecewise constant-in-time vector field such that
\begin{equation}\label{eq:sht_def}
\sht|_{I_n} \coloneqq \sum_{\ver\in\Ver} \sht^{\ver,n} \quad \forall n\in \{1,\dots,N\}.
\end{equation}
Then, following the same argument as in~\cite[Theorem~4.2]{ESV2017}, we have the equilibration identity
\begin{equation}\label{eq:sht_equilib}
\p_t \Uht + \nabla \cdot \sht = f_{h,\tau} \quad\text{a.e.\ in } \Omega\times (0,T).
\end{equation}

\section{Main results}

We now give the main results on the \emph{a posteriori} error bound for the energy norm of the error. In Theorem~\ref{thm:energy_main_bound}, we show that the energy norm of the error $u-\ouht$ can be bounded, up to data oscillation, in terms of computable estimators derived from $\uht-\Uht$ and $\sht+\nabla \ouht$. Furthermore, the resulting estimators are globally efficient, as shown in Theorem~\ref{thm:lower_bound} below.

\begin{theorem}[Upper bound]\label{thm:energy_main_bound}
We have the upper bound
\begin{equation}\label{eq:energy_main_upper}
\norm{u-\ouht}_E \leq \left(\int_0^T \left( \frac{1}{4}\norm{\nabla(\uht-\Uht)}_\Omega^2+\norm{\sht+\nabla\ouht}_{\Omega}^2\right)\mathrm{d}t\right)^{\frac{1}{2}} + \etaOscEth,
\end{equation}
where the data oscillation term $\etaOscEth$ is defined by
\begin{equation}\label{eq:def_estaosceth}
\etaOscEth \coloneqq \sup_{\varphi\in Y\setminus\{0\}}\frac{\int_0^T\pair{f-f_{h,\tau}}{\varphi} \mathrm{d}t + (u_0-u_{h,\tau,0},\varphi(0))_{\Omega}}{\normYs{\varphi}}.
\end{equation}
\end{theorem}
\begin{proof}
Theorem~\ref{thm:Z_infsup} implies that
\begin{equation}
\norm{u-\ouht}_E = \sup_{\varphi\in Y\setminus\{0\}}\frac{\pair{\mathcal{R}(\ouht)}{\varphi}_{Y^*\times Y}}{\normYs{\varphi}},
\end{equation}
where $Y^*$ denotes the dual space of $Y$ and the residual $\mathcal{R}(\ouht)$ is defined by
\begin{multline}\label{eq:energy_main_1}
\pair{\mathcal{R}(\ouht)}{\varphi}_{Y^*\times Y}\coloneqq \int_0^T \pair{f}{\varphi}\mathrm{d}t + (u_0,\varphi(0))_\Omega
\\  - (\ouht(T),\varphi(T))_{\Omega}
-\int_0^T\left[-\pair{\p_t\varphi}{\ouht}+(\nabla \ouht,\nabla\varphi)_\Omega \right] \mathrm{d}t,
\end{multline}
for all $\varphi\in Y$.
Using the flux equilibration identity~\eqref{eq:sht_equilib}, integration-by-parts, and the identities $\ouht(T)=\Uht(T)$ and $\Uht(0)=u_{h,\tau,0}$, we see that, for any $\varphi\in Y$,
\begin{equation}\label{eq:energy_main_2}
\int_0^T (f_{h,\tau},\varphi)_\Omega\mathrm{d}t = (\ouht(T),\varphi(T))_\Omega - (u_{h,\tau,0},\varphi(0))_{\Omega} + \int_0^T \left[-\pair{\p_t\varphi}{\Uht}-(\sht,\nabla \varphi)_\Omega \right]\mathrm{d}t.
\end{equation}
Therefore, the combination of~\eqref{eq:energy_main_1} with \eqref{eq:energy_main_2} implies that
\begin{multline}
\pair{\mathcal{R}(\ouht)}{\varphi}_{Y^*\times Y} = \int_0^T \pair{f-f_{h,\tau}}{\varphi}\mathrm{d}t + (u_0-u_{h,\tau,0},\varphi(0))_{\Omega}
\\ + \int_0^T\left[\pair{\p_t\varphi}{\ouht - \Uht} - (\sht+\nabla \ouht,\nabla \varphi)_{\Omega} \right]\mathrm{d}t.
\end{multline}
The triangle inequality and the Cauchy--Schwarz inequality then imply that
\begin{equation}
\sup_{\varphi\in Y\setminus\{0\}}\frac{\pair{\mathcal{R}(\ouht)}{\varphi}_{Y^*\times Y}}{\normYs{\varphi}} \leq \left(\int_0^T \left( \norm{\nabla(\ouht-\Uht)}_\Omega^2+\norm{\sht+\nabla\ouht}_{\Omega}^2\right)\mathrm{d}t\right)^{\frac{1}{2}}+\etaOscEth,
\end{equation}
which shows the upper bound~\eqref{eq:energy_main_upper} upon noting that $\ouht-\Uht=\frac{1}{2}(\uht-\Uht)$.
\end{proof}

We now state the lower bound for the error in terms of the estimators, showing global efficiency of the estimators under the additional condition of the form $h^2\lesssim \tau$. In the following, let $h_\oma \coloneqq \diam \oma$ denote the diameter of the vertex patch $\oma$ for each $\ver\in\Ver$. We restrict ourselves here to the case where the space dimension $\sdim\leq 3 $ as this result makes use of some existing results in~\cite{ESV2019}.

\begin{theorem}[Lower bound]\label{thm:lower_bound}
If $1\leq \sdim \leq 3 $ and if there exists a constant $\gamma>0$ such that $h_{\oma}^2 \leq \gamma \tau_n$ for all $\ver\in\Ver$ and all $n\in\{1,\dots,N\}$, then we also have the global lower bound
\begin{equation}\label{eq:energy_main_lower}
\int_0^T \left( \frac{1}{4}\norm{\nabla(\uht-\Uht)}_\Omega^2+\norm{\sht+\nabla\ouht}_{\Omega}^2\right)\mathrm{d}t \lesssim \norm{u-\ouht}_E^2 + \left[\etaOscEthGlobal\right]^2,
\end{equation}
where the data oscillation term $\etaOscEthGlobal$ is defined by
\begin{subequations}
\begin{gather}
\left[\etaOscEthGlobal\right]^2\coloneqq \left[\wetaOscEth\right]^2
+ \sum_{n=1}^N\sum_{\ver\in\Ver}\left[\eta_{\mathrm{osc}}^{\ver,n}\right]^2,
\\
\wetaOscEth \coloneqq \sup_{\substack{ \varphi_h \in H^1(0,T;\VTp)\setminus\{0\} \\ \varphi_h(0)=0 }}\frac{\int_0^T\pair{f-f_{h,\tau}}{\varphi_h}\mathrm{d}t}{\normYs{\varphi_h}}, \qquad
\left[\eta_{\mathrm{osc}}^{\ver,n}\right]^2 \coloneqq \int_{I_n} \norm{f-f_{h,\tau}}_{H^{-1}(\oma)}^2\mathrm{d}t.
\label{eq:data_osc_2}
\end{gather}
\end{subequations}
The hidden constant in~\eqref{eq:energy_main_lower} depends only on the shape-regularity of~$\T$, the dimension~$\sdim$,  the constant~$\gamma$, and on the constant~$C_{\Pi}$ appearing in~\eqref{eq:L2_stab_ortho}.
\end{theorem}

The proof of Theorem~\ref{thm:lower_bound} is the subject of Section~\ref{sec:proof_of_the_lower_bound} below.
Theorems~\ref{thm:energy_main_bound} and~\ref{thm:lower_bound} show global upper and lower bounds on the error. 
Therefore, the error estimator gives a guaranteed upper bound on the error up to data oscillation, i.e.\ there are no unknown constants in the upper bound. 
Furthermore, the error estimators are also globally efficient.
However, the local efficiency is less clear; essentially this is due to the fact that the analysis of the jump estimator $\norm{\uht-\Uht}_X$ hinges on some nonlocal arguments both in space and in time.
Note that the locality of the efficiency bounds of the estimator $\norm{\uht-\Uht}_X$ is also a challenge in the setting of $Y$-norm error estimators for $u-\Uht$, c.f.\ \cite{Verfurth2003}, where the lower bounds are global in space.

\begin{remark}[Alternative error quantifiers]\label{rem:alternative_error_measures}
Various alternative error quantifiers have appeared in the literature in order to handle the problem of efficiency of the estimators. Often these quantifiers combine the norms of the errors for the two reconstructions $\uht$ and $\Uht$ simultaneously, see for instance~\cite{AkrivisMakridakisNochetto2009,MakridakisNochetto2006,Schotzau2010}. In particular, local efficiency of the estimators in space and time was shown in~\cite{ESV2019} for the setting of $L^2(H^1)$ norm.
In the context of the energy norm, we can consider the quantity $\mathcal{E}_{h,\tau}\coloneqq \norm{u-\uht}_E + \norm{u-\Uht}_E$, in which case the analysis above and the results in~\cite{ESV2019} imply that, up to data oscillation terms, we have some form of global equivalence between $\mathcal{E}_{h,\tau}$ and $\norm{u-\ouht}_E$, at least under the hypotheses of Theorem~\ref{thm:lower_bound}.
This offers a further perspective on the global nature of these alternative error quantifiers that have been considered in the literature, although the localization of these error measures might still differ. 
\end{remark}

\section{Proof of the lower bound}\label{sec:proof_of_the_lower_bound}

The starting point for the proof of the lower bound is to bound the temporal jump estimator in terms of the energy norm of the error. To this end, we derive an identity for the residual when the test functions are restricted to the subspace $H^1(0,T;\VTp)\subset Y$, which consists of functions that are $H^1$ w.r.t\ time with values in $\VTp$. 

\begin{lemma}[Identity for the difference between reconstructions]\label{lem:main_identity}
For all $\varphi_h \in H^1(0,T;\VTp)$, we have the identity
\begin{multline}\label{eq:jump_discrete_energy_stab_3}
 \frac{1}{2}\int_{0}^{T} \left[(\partial_t \varphi_h ,\uht-\Uht)_\Omega+(\nabla \varphi_h,\nabla(\uht-\Uht))_\Omega  \right]\mathrm{d}t  \\ = \Bz(\bm{u}-\overline{\bm{u}}_{h,\tau},\varphi_h) -  \int_{0}^{T} \pair{f-f_{h,\tau}}{\varphi_h}\mathrm{d}t - (u_0-u_{h,\tau,0},\varphi_h(0))_\Omega, 
\end{multline}
where the bilinear form $\Bz$ is defined above in~\eqref{eq:B_bilinear}, and where $\bm{u}, \overline{\bm{u}}_{h,\tau}\in \bm{Z}$ are defined by $\bm{u}\coloneqq (u,u(T))$ and $\overline{\bm{u}}_{h,\tau}\coloneqq (\ouht,\ouht(T))$.
\end{lemma}
\begin{proof}
Note that $H^1(0,T;\VTp)\subset Y$ since $\VTp\subset H^1_0(\Omega)$.
Using the identity $\ouht(t_n)=\Uht(t_n)=u_{h,\tau,n}$ for all $n\in \{0,\dots,N\}$ and using~\eqref{eq:IE_FEM_2} (which holds in a pointwise sense in time), we find that 
\begin{equation}\label{eq:jump_discrete_energy_stab_4}
(\ouht(T),\varphi_h(T))_\Omega +\int_{0}^{T} \left[-\pair{\p_t \varphi_h}{\Uht}+(\nabla \uht,\nabla \varphi_h)_{\Omega}\right]\mathrm{d}t = 
(u_{h,\tau,0},\varphi(0))_\Omega  + \int_{0}^{T} (f_{h,\tau},\varphi_h)_\Omega\mathrm{d}t,
\end{equation}
where we have used integration-by-parts on the temporal derivative term.
After addition-subtraction of the terms in~\eqref{eq:jump_discrete_energy_stab_4} above, we find that
\begin{multline*}
\Bz(\overline{\bm{u}}_{h,\tau},\varphi_h)
 = (\ouht(T),\varphi_h(T))_{\Omega}+\int_{0}^{T} \left[-(\p_t \varphi_h,\ouht)_{\Omega}+(\nabla \ouht, \varphi_h)_\Omega\right]\mathrm{d}t
\\ 
=(u_{h,\tau,0},\varphi_h(0))_\Omega + \int_{0}^{T}\left[ (f_{h,\tau},\varphi_h)_\Omega + ( \p_t \varphi_h, \Uht - \ouht )_{\Omega} + (\nabla \varphi_h, \nabla (\ouht - \uht) )_{\Omega} \right] \mathrm{d}t
\\ 
=(u_{h,\tau,0},\varphi_h(0))_\Omega + \int_{0}^{T}(f_{h,\tau},\varphi_h)_\Omega \mathrm{d}t - \frac{1}{2}\int_{0}^{T} \left[(\p_t\varphi_h,\uht-\Uht)_{\Omega}+(\nabla \varphi_h,\nabla (\uht-\Uht))_{\Omega} \right]\mathrm{d}t,
\end{multline*}
where we have simplified $\ouht-\Uht = \uht-\ouht = \frac{1}{2}\left(\uht-\Uht\right)$.
Recalling \eqref{eq:Z_weakform_2}, after subtracting the identity above from $\Bz(\bm{u},\varphi_h)$, we find that
\begin{multline}
\Bz(\bm{u}-\overline{\bm{u}}_{h,\tau},\varphi_h)
 = ( u_0-\ouht(0),\varphi_h(0))_\Omega +
\int_{0}^{T}\pair{f-f_{h,\tau}}{\varphi_h}\mathrm{d}t 
\\ + \frac{1}{2}\int_{0}^{T} \left[(\p_t\varphi_h,\uht-\Uht)_{\Omega}+(\nabla \varphi_h,\nabla (\uht-\Uht))_{\Omega} \right]\mathrm{d}t,
\end{multline}
which is equivalent to~\eqref{eq:jump_discrete_energy_stab_3}.
\end{proof}

We now show that the temporal jump estimator is bounded by the energy norm of the error.

\begin{lemma}[Efficiency of the jump estimator]\label{lem:jump_discrete_energy_stab}
We have the bound
\begin{equation}\label{eq:jump_discrete_energy_stab}
\norm{\uht-\Uht}_X \lesssim \norm{u-\overline{u}_{h,\tau}}_{E} + \wetaOscEth,
\end{equation}
where the hidden constant depends only on $C_{\Pi}$, and where~$\wetaOscEth$ is defined in~\eqref{eq:data_osc_2}.
\end{lemma}

\begin{proof}
Since $\VTp$ is finite dimensional, we can find a unique $\varphi_h \in H^1(0,T;\VTp)\cap C([0,T];\VTp)$ that satisfies $\varphi_h(0)=0$ and, for a.e.\ $t\in (0,T)$,
\begin{equation}\label{eq:dual_problem_definition}
\begin{aligned}
(\p_t\varphi_h(t),v_h)_\Omega + (\nabla \varphi_h(t),\nabla v_h)_{\Omega} &= (\nabla (\uht-\Uht)(t),\nabla v_h)_{\Omega} && \forall v_h \in V_h.
\end{aligned}
\end{equation}
Note that $\varphi\in Y$ since $\VTp\subset H^1_0(\Omega)$.
Furthermore, we can define $z_h\in L^2(0,T;\VTp)$ as the unique solution of $(\nabla z_h(t),\nabla v_h)_\Omega=(\p_t \varphi_h(t),\nabla v_h)_\Omega$ for all $v_h\in \VTp$. Then, we have the identities $(\nabla(z_h(t)+\varphi_h(t)),\nabla v_h)_\Omega=(\nabla(\uht-\Uht))(t),\nabla v_h)_{\Omega}$ for a.e.\ $t\in(0,T)$, which implies $z_h+\varphi_h=\uht-\Uht$ in $\Omega$, for a.e.\ $t\in (0,T)$.
Therefore, expanding the square, we obtain
\begin{equation}\label{eq:jump_discrete_energy_stab_1}
\begin{aligned}
\int_0^T\norm{\nabla(\uht-\Uht)}_{\Omega}^2\mathrm{d}t & = 
 \int_0^T \left[(\p_t \varphi_h ,\uht-\Uht)_\Omega+(\nabla \varphi_h,\nabla(\uht-\Uht))_\Omega  \right]\mathrm{d}t
\\ & = \int_0^T \norm{\nabla(z_h+\varphi_h)}_{\Omega}^2\mathrm{d}t 
\\ & = \int_0^T \left(\norm{\nabla z_h}_\Omega^2 + \norm{\nabla\varphi_h}_{\Omega}^2\right)\mathrm{d}t + 2\int_0^T (\nabla z_h,\nabla \varphi_h)_\Omega \mathrm{d}t
\\ & = \int_0^T \left(\norm{\partial_t \varphi_h}_{\VTp^*}^2+\norm{\nabla \varphi_h}_\Omega^2\right)\mathrm{d}t + 2\int_\Omega (\p_t \varphi_h,\varphi_h)_\Omega\mathrm{d}t
\\ & =\int_0^T \left(\norm{\partial_t \varphi_h}_{\VTp^*}^2+\norm{\nabla \varphi_h}_\Omega^2\right)\mathrm{d}t + \norm{\varphi_h(T)}_{\Omega}^2,
\end{aligned}
\end{equation}
where, in passing to the last line above, we have used the fact that $\varphi_h(0)=0$. It then follows from~\eqref{eq:DiscreteNegnorm_stab} that $\norm{\p_t \varphi_h}_{H^{-1}(\Omega)}\leq C_\Pi \norm{\p_t \varphi_h}_{\VTp^*}$ and thus
\begin{equation}\label{eq:jump_discrete_energy_stab_2}
\normYs{\varphi_h}^2 = \int_0^T \left(\norm{\partial_t \varphi_h}_{H^{-1}(\Omega)}^2+\norm{\nabla \varphi_h}_\Omega^2\right)\mathrm{d}t + \norm{\varphi_h(T)}_{\Omega}^2 \lesssim \norm{\uht-\Uht}_X^2,
\end{equation}
for some constant depending only on $C_{\Pi}$.
Combining the identity~\eqref{eq:jump_discrete_energy_stab_3} of Lemma~\ref{lem:main_identity} (noting that $\varphi_h(0)=0$) along with the first line of \eqref{eq:jump_discrete_energy_stab_1} above, it is found that
\begin{equation}
  \frac{1}{2}\norm{\uht-\Uht}_X^2  = \Bz(\bm{u}-\overline{\bm{u}}_{h,\tau},\varphi_h) - \int_{0}^{T} \pair{f-f_{h,\tau}}{\varphi_h}\mathrm{d}t.
\end{equation}
The triangle inequality and Theorem~\ref{thm:Z_infsup} then imply that
\begin{equation}
\norm{\uht-\Uht}_X^2 \leq \left(\norm{u-\uht}_E+\wetaOscEth \right)\normYs{\varphi_h}.
\end{equation}
Thus~\eqref{eq:jump_discrete_energy_stab} readily follows from the inequality above and from~\eqref{eq:jump_discrete_energy_stab_2}.
\end{proof}

In order to bound the equilibrated flux estimator in terms of the error, we will use here part of~\cite[Theorem~5.1]{ESV2019}, which showed that the flux estimator is efficient \emph{up to the temporal jump terms} with respect to the $X$-norm of the error $u-\uht$, under the condition $h^2\lesssim \tau$. We also restrict here ourselves to the case $1\leq \sdim\leq 3$ as this was also assumed in~\cite{ESV2019}.

\begin{lemma}[\cite{ESV2019}]\label{lem:X_norm_lower_bounds}
Suppose that $1\leq \sdim \leq 3 $, and suppose that there exists constant $\gamma>0$ such that $h_{\oma}^2\leq \gamma \tau_n$ for every $\ver\in\Ver$ and every $n\in\{1,\dots,N\}$.
Then, for each $K\in\T$ and each $n\in\{1,\dots,N\}$, we have
\begin{equation}\label{eq:X_norm_local_lower_bound}
\int_{I_n}\norm{\sht+\nabla\uht}_K^2\mathrm{d}t \lesssim \sum_{\ver\in\Ver_K} \left[\int_{I_n}\left(\norm{\nabla(u-\uht)}_{\oma}^2+\norm{\nabla(\uht-\Uht)}_{\oma}^2\right)\mathrm{d}t+\left[\eta_{\mathrm{osc}}^{\ver,n}\right]^2\right],
\end{equation}
and also
\begin{equation}\label{eq:X_norm_global_lower_bound}
\int_0^T\norm{\sht+\nabla \uht}_{\Omega}^2\mathrm{d}t \lesssim \norm{u-\uht}_X^2+\norm{\uht-\Uht}_X^2 + \sum_{n=1}^N\sum_{\ver\in\Ver}\left[\eta_{\mathrm{osc}}^{\ver,n}\right]^2,
\end{equation}
where the hidden constant depends only on the shape-regularity of $\T$, the dimension $\sdim$ and on $\gamma$, and where the data oscillation~$\eta_{\mathrm{osc}}^{\ver,n}$ is defined in~\eqref{eq:data_osc_2} above.
\end{lemma}

\paragraph{Proof of Theorem~\ref{thm:lower_bound}.}
To prove~\eqref{eq:energy_main_lower}, we start by noting that, by the triangle inequality, 
\begin{equation}
\int_{0}^T \norm{\sht+\nabla\ouht}_{\Omega}^2\mathrm{d}t \lesssim \int_{0}^T \left(  \norm{\sht+\nabla \uht}_\Omega^2+ \norm{\nabla(\uht-\Uht)}_{\Omega}^2 \right) \mathrm{d}t,
\end{equation}
where we also used $\ouht-\Uht = \frac{1}{2}(\uht-\Uht)$.
Applying Lemma~\ref{lem:X_norm_lower_bounds}, in particular~\eqref{eq:X_norm_global_lower_bound}, we then find that
\begin{equation}
\begin{split}
\int_0^T \norm{\sht+\nabla \ouht}^2_\Omega\mathrm{d}t & \lesssim \norm{u-\uht}^2_{X} + \norm{\uht-\Uht}_X^2 + \sum_{n=1}^N\sum_{\ver\in\Ver}\left[\eta_{\mathrm{osc}}^{\ver,n}\right]^2
\\ & \lesssim \norm{u-\ouht}_X^2 + \norm{\uht-\Uht}_X^2 + \sum_{n=1}^N\sum_{\ver\in\Ver}\left[\eta_{\mathrm{osc}}^{\ver,n}\right]^2,
\end{split}
\end{equation}
where we have used again the triangle inequality and the identity $\ouht-\uht = \frac{1}{2}(\Uht-\uht)$ in passing to the second-line above.
We then use Lemma~\ref{lem:jump_discrete_energy_stab} above to bound $\norm{\uht-\Uht}_X$, which gives
\begin{equation}
\int_0^T \norm{\sht+\nabla \ouht}^2_\Omega \mathrm{d}t+\norm{\uht-\Uht}_X^2 \lesssim 
\norm{u-\ouht}_{E}^2 + [\wetaOscEth]^2 + \sum_{n=1}^N\sum_{\ver\in\Ver}\left[\eta_{\mathrm{osc}}^{\ver,n}\right]^2,
\end{equation}
which shows~\eqref{eq:energy_main_lower}.\hfill\proofbox 

\begin{remark}[Connection to the semi-discrete setting]
To make a connection to the discussion in the introduction regarding the semi-discrete setting, it is possible to give an alternative proof of the hypercircle theorem~\eqref{eq:hypercircle_2} using entirely analogous arguments as those above. In particular, in the semi-discrete setting with vanishing data oscillation, it is found that
\begin{equation}
\Bz(\bm{u}-\bm{\overline{u}}_{\tau},\varphi) = \frac{1}{2}\int_0^T \left( \pair{\p_t \varphi}{u_\tau-U_\tau}+(\nabla (u_\tau-U_\tau),\nabla \varphi)_\Omega \right)\mathrm{d}t\quad \forall \varphi \in Y,
\end{equation}
from which \eqref{eq:hypercircle_2} follows immediately in view of Theorem~\ref{thm:Z_infsup} and the invariance of $\normYs{\cdot}$ with respect to reversal of the time direction.
\end{remark}

\appendix 

\section{An example}\label{sec:counterexamples}

We give a brief example to illustrate how the temporal jump estimator is not generally efficient with respect to either $u-\uht$ or $u-\Uht$ in the energy norm setting.
Indeed, to see this, it is enough to consider the case of an ordinary differential equation: find $u\colon [0,1]\rightarrow \R$ such that 
$$
\p_t u + \lambda u =1 \quad \text{in }(0,1], \quad u(0)=0,
$$
where $\lambda>0$ is a parameter.
This problem can be thought of as arising from some parabolic PDE after a transformation to Fourier modes, and suitable scaling of the parameters.
Note that in this case the source term~$f=1$ is constant, so there is no data oscillation.
The exact solution is $u(t)=\frac{1-\mathrm{e}^{-\lambda t}}{\lambda}$ for all $t\in [0,1]$. 
Consider its discretization by the implicit Euler method on a single time-step of length $\tau_1=1$.
Thus we find that $u_\tau(t) = \frac{1}{1+\lambda}$ for all $t\in(0,1]$ and $U_\tau(t) = \frac{t}{1+\lambda}$ for all $t\in [0,1]$. 
Then, direct computations show that
\begin{equation}
\lim_{\lambda\tends 0} \frac{\norm{u_\tau-U_\tau}_E}{\norm{u-U_\tau}_E} = \infty,
\end{equation}
so the estimator is not efficient with respect to the energy norm of the error $u-U_\tau$ when the parameter $\lambda$ is small. 
It is also found that
\begin{equation}
\lim_{\lambda\tends \infty}\frac{\norm{u_\tau-U_\tau}_E}{\norm{u-u_\tau}_E} = \infty,
\end{equation}
so the estimator is not efficient with respect to the energy norm of the error $u-u_\tau$ when the parameter $\lambda$ is large.
More generally, the example above shows that we cannot generally expect to show the efficiency of the estimator with regards to the energy norm error for either of these reconstructions, independently of discretization and problem parameters.
In particular, the loss of efficiency occurs in cases where one reconstruction is significantly closer to the true solution than the other. 
Incidentally, this also shows that there is no general answer to which reconstruction is the more accurate one.

\bibliographystyle{siamplain_noURL}
\bibliography{newbiblio}

\end{document}